\documentclass[a4paper,10pt]{amsart}
 \usepackage{amsfonts,mathrsfs}
 \usepackage{amsthm}
 \usepackage{todonotes}
 \usepackage{amssymb}
 \usepackage{enumerate}
 \usepackage{tikz-cd}
\usepackage{cleveref}
\theoremstyle{definition}
\newtheorem{Def}{Definition}[section]
\newtheorem{ex}[Def]{Example}
\newtheorem{rem}[Def]{Remark}
\theoremstyle{plain}
\newtheorem{prop}[Def]{Proposition}
\newtheorem{thm}[Def]{Theorem}
\newtheorem*{thm*}{Theorem}
\newtheorem{lem}[Def]{Lemma}
\newtheorem{cor}[Def]{Corollary}
\newtheorem*{cor*}{Corollary}

\newtheorem*{con*}{Conjecture}

\newtheorem*{verm*}{Vermutung}

\usepackage{mathrsfs}

\newcommand{\bA}{\mathfrak A}
\newcommand{\bB}{\mathfrak B}
\newcommand{\bC}{\mathfrak C}
\newcommand{\bD}{\mathfrak D}

\newcommand{\conv}{\operatorname{conv}}

\newcommand{\Sper}{\operatorname{Sper}}

\newcommand{\GL}{\operatorname{{\mathbf{GL}}}}

\newcommand{\cH}{{\mathcal H}}

\newcommand{\fA}{{\mathfrak A}}

\newcommand{\fo}{{\mathfrak o}}

\newcommand{\C}{{\mathbb C}}
\newcommand{\R}{{\mathbb R}}

\newcommand{\Q}{{\mathbb Q}}
\newcommand{\N}{{\mathbb N}}
\newcommand{\Z}{{\mathbb Z}}

\newcommand{\Pol}{\operatorname{P}}

\title[Spectrahedral Shadows and Completely
Positive Maps]{Spectrahedral Shadows and Completely
Positive Maps on Real Closed Fields}
\author{Manuel Bodirsky}
\address{Technische Universit\"at, Dresden, Germany} 
\email{manuel.bodirsky@tu-dresden.de}
\author{Mario Kummer}
\address{Technische Universit\"at, Dresden, Germany} 
\email{mario.kummer@tu-dresden.de}
\author{Andreas Thom}
\address{Technische Universit\"at, Dresden, Germany} 
\email{andreas.thom@tu-dresden.de}
\thanks{Manuel Bodirsky received funding from the ERC under the European Community's Seventh Framework Programme (Grant Agreement no. 681988, CSP-Infinity). Mario Kummer was partially supported by DFG grant 421473641. Andreas Thom was partially supported by ERC Consolidator Grant 681207 and DFG SPP 2026.}

\newcommand{\comment}[1]{}

\begin{document}


\begin{abstract}
 In this article we develop new methods for exhibiting convex semialgebraic sets that are not  spectrahedral shadows. We characterize when the set of nonnegative polynomials with a given support is a spectrahedral shadow in terms of sums of squares. As an application of this result we prove that the cone of copositive matrices of size $n\geq5$ is not a spectrahedral shadow, answering a question of Scheiderer. Our arguments are based on the model theoretic observation that any formula defining a spectrahedral shadow must be preserved by every unital $\R$-linear completely positive map $R\to R$ on a real closed field extension $R$ of $\R$.
\end{abstract}
\maketitle

\section{Introduction}
A set $S \subseteq {\mathbb R}^n$ is called a 
\emph{spectrahedron} 
if 
$$S = \{ (x_1,\dots,x_n) \mid A_0 + A_1 x_1 + \dots + A_n x_n \succeq 0\}$$
for some $k \times k$ real symmetric matrices $A_0,A_1,\dots,A_n$. Here $A\succeq0$ says that the real symmetric matrix $A$ is positive semidefinite, and the expression
$A_0 + A_1 x_1 + \cdots + A_n x_n \succeq 0$ is called a \emph{linear matrix inequality (LMI) of size $k$}. 
Note that spectrahedra are convex and \emph{semialgebraic}, i.e., 
definable over ${\mathbb R}$ by a first-order formula in the language of real-closed fields (i.e., the language of ordered rings). 

There has been ample interest in representing convex semialgebraic sets $K\subseteq\R^n$ as a \emph{spectrahedral shadow}. This amounts {to} finding real symmetric matrices $A_0,\ldots,A_n,B_1,\ldots,B_m$ of the same size such that $$K=\{x\in\R^n\mid \exists y\in\R^m\colon A_0+x_1A_1+\cdots+x_nA_n+y_1B_1+\cdots{+} y_mB_m\succeq0\}.$$
The interest stems primarily from the fact that there are efficient methods for approximately optimizing a linear function over such a set. 
It was shown by Helton and Nie \cite{heltonnie1} that every convex semialgebraic set satisfying certain smoothness and curvature conditions is a spectrahedral shadow. Moreover, Scheiderer proved that every  convex semialgebraic set in the plane is a spectrahedral shadow \cite{claus2}. On the other hand, not every convex semialgebraic set is a spectrahedral shadow. The first counterexamples were constructed by Scheiderer \cite{heltonniefalse}. Further examples were found in \cite{fawzi2019set, bettiol2019convex}, both using similar techniques as Scheiderer in his seminal work.

Recent efforts also focus on finding a representation as spectrahedral shadow with several LMIs of a bounded size. 
For instance Averkov~\cite{Ave} showed that the cone of $k\times k$ real symmetric positive semidefinite matrices has no semidefinite extended representation 
with LMIs of size less than $k$ (for a formal definition, see \Cref{sect:k-pos}). 
A generalization of this result in terms of convex geometry was given by Saunderson \cite{MR4080391}. On the positive side, Scheiderer \cite{MR4234144} proved that for every closed convex semialgebraic set $K\subseteq\R^2$ LMIs of size two are sufficient.

Our approach to these questions is based on an easy consequence of the Tarski principle for real closed fields. The question whether a given semialgebraic set $K \subseteq {\mathbb R}^n$ is a spectrahedral shadow is equivalent to the question whether $K$ has a primitive positive definition over the structure $\mathfrak A$ with domain $\mathbb R$ and with a relation $S \subseteq {\mathbb R}^n$ for every spectrahedron $S$.
By the Tarski principle for real closed fields, we may
equivalently study this question over real-closed field extensions $R$ of $\mathbb R$ instead of $\mathbb R$. 
More precisely, if $\phi$ is a formula in the language of real closed fields  that defines $K$,
let $K^*$ be the set defined by $\phi$ over $R$. Likewise, 
let ${\mathfrak A}^*$ denote 
the structure with domain $R$ which contains for every formula $\psi$ that defines a spectrahedron over ${\mathbb R}$
the relation defined by $\psi$ over $R$. 
Then $K$ has a primitive positive definition in ${\mathfrak A}$ if and only if $K^*$ has a primitive positive definition in $\mathfrak A^*$. 
A map $h \colon R \to R$ \emph{preserves} a relation $K \subseteq R^n$ 
if $(h(a_1),\dots,h(a_n)) \in K$ whenever $(a_1,\dots,a_n) \in K$.
We use the easy fact that $K^*$ does not have an existential positive definition in ${\mathfrak A}^*$ (and hence in particular no primitive positive definition) if
there exists a map $h \colon R \to R$ that preserves all the relations of ${\mathfrak A}^*$ but not $K^*$. In \Cref{app:model} we discuss this approach in more detail and prove a converse of the last statement.

We show that a map $L \colon R \to R$ preserves ${\mathfrak A}^*$ if and only if it is unital, $\R$-linear, and completely positive in the sense of Operator theory. Thus, in order to prove that $K$ is not a spectrahedral shadow, it suffices to exhibit a unital, $\R$-linear, and completely positive map $L\colon R \to R$ on a real closed field extension $\R\subseteq R$ that does not preserve the set defined by $\phi$.
In this way we prove that the set of copositive matrices of size $n$ is not a spectrahedral shadow whenever $n\geq 5$, a family of convex semialgebraic sets that resisted Scheiderer's efforts, see \cite[Question~5.2]{heltonniefalse}.
A real symmetric matrix $A$ is \emph{copositive} if $x^{\top} A x \geq 0$ for every non-negative real vector $x$.  
Note that for $n\leq4$  it follows from the main result in \cite{dian} that this set is a spectrahedral shadow. 

Note that convexity is not only preserved by conjunction (i.e., intersection) and existential quantification (i.e., projections), but also by universal quantification. 
In the light of the results mentioned above one may ask whether every convex semi-algebraic set has a definition using a formula from \emph{positive first-order logic} over the structure formed by all spectrahedra. Positive first-order logic is the extension of existential positive logic by universal quantification (so only negation from first-order logic is prohibited). 
Our proof shows that the considered convex semialgebraic sets are not definable from spectrahedra even in this more expressive logic. The reason is that the map $L \colon R \to R$ that we construct is  \emph{surjective}. 
Hence, we may combine the Tarski principle with the easy direction of Lyndon's preservation theorem (see, e.g., Corollary 10.3.5 in~\cite{HodgesLong}) 
to conclude that for example the set of copositive matrices does not have a definition from spectrahedra in positive first-order logic.

Apart from the Tarski principle the only ingredient to our proof is a separation theorem (\Cref{thm:seplin}) that follows from the results in  \cite{rcsep}. We will give a self-contained proof in \Cref{sec:appendi}.
 In particular, we do not make use of the theory developed in \cite{heltonniefalse}. Our result on copositive matrices is a consequence of the following more general result:
\begin{thm*}[see \Cref{thm:main} below]
 Let $S\subset\Z_{\geq0}^n$ a finite set and let $\Pol_+(S)$ the set of all nonnegative polynomials
 with coefficients from ${\mathbb R}$
 whose support is contained in $S$. Then the following are equivalent:
 \begin{enumerate}
  \item $\Pol_+(S)$ is a spectrahedral shadow.
  \item There exists an integer $d>0$ such that $p(x_1^d,\ldots,x_n^d)$ is a sum of squares for all $p\in\Pol_+(S)$.
 \end{enumerate}
\end{thm*}
Finally, in \Cref{sec:ave} we reprove Averkov's above mentioned result in a similar fashion. Here we use results about positive definite maps \cite{Karlin}.

We now give a brief outline of our proof strategy. We first observe in \Cref{sect:k-pos} that a map $L \colon R\to R$ for some real closed field extension $R$ of $\R$ preserves {$\fA^*$} if and only if $L$ is unital, $\R$-linear, and completely positive in the sense that applying $L$ entry-wise to a positive semidefinite symmetric matrix over $R$ gives a positive semidefinite matrix. When $R=R'[[\epsilon^\Lambda]]$ is the field of Hahn series with value group $\Lambda$ and coefficients in another real closed field $R'$, we consider maps $L_f \colon R\to R$ of the form $L_f \left (\sum_{a\in\Lambda}c_a\epsilon^a \right )=\sum_{a\in\Lambda}f(a)c_a\epsilon^a$ for functions $f \colon \Lambda\to R'$. On the other hand, extending $f$ linearly, we obtain a linear functional $T_f \colon R'[\Lambda]\to R'$ on the group ring of $\Lambda$. The crucial insight of \Cref{sect:k-pos} is that $L_f$ is completely positive whenever $T_f$ is positive on every nonzero square in $R'[\Lambda]$. In \Cref{sect:groupring} we focus on the case $\Lambda=\Q^n$. Elements of the group ring $R'[\Q^n]$ can be regarded as functions on $(R'_{>0})^n$. If such a function takes only nonnegative values but is not a sum of squares in $R'[\Q^n]$, then it can be separated by a linear functional $T_f$ from the cone of sums of squares after a base change to a larger real closed field by \Cref{thm:seplin}. The corresponding map $L_f$ is then completely positive but does not preserve a certain convex cone of nonnegative polynomials. Applying this construction to the so-called Horn polynomial proves that the cone of copositive matrices of size $n\geq 5$ is not a spectrahedral shadow.

\bigskip

\noindent \textbf{Acknowledgements.}
We would like to thank Claus Scheiderer for helpful comments on a previous version of the manuscript.

\section{$k$-positive maps}
\label{sect:k-pos}
In this section let $R$ always be a real closed field extension of $\R$.
Let $\Gamma_k$ be the set of all first-order formulas in the language of real closed fields with parameters from $\R$ that express 
a linear matrix inequality of size $k$ with coefficients from $\R$, 
and let $\Gamma=\cup_{k=1}^\infty\Gamma_k$. Let ${\mathfrak A}_k$ be the relational structure with domain $\R$ whose relations consist of the sets defined by the formulas in $\Gamma_k$. We follow standard terminology in model theory, see, e.g.,~\cite{HodgesLong,Tent-Ziegler}. For the convenience of the reader, we have collected the relevant concepts in Appendix~\ref{sect:basic}. The signature of $\bA_k$ should not be confused with the signature of real closed fields. If a set $K$ has a primitive positive definition over ${\mathfrak A}_k$ then we say that $K$ has a \emph{semidefinite extended formulation with LMIs of size $k$}, following the terminology of Averkov~\cite{Ave}.

We say that a map $h \colon R \to R$ \emph{preserves} a formula in the language of real closed fields if it preserves the relation defined by the formula over $R$. 
The map $h$ is called \emph{unital} if it preserves the formula $x=1$, and for $R' \subseteq R$ it is called \emph{$R'$-linear} if it preserves the formula $z= \lambda \cdot x + \mu \cdot y$ for all $\lambda,\mu \in R'$. 

\begin{Def}
 Let $k\in\Z_{>0}$. A map $L\colon R \to R$ is called \emph{$k$-positive} if for every $k\times k$ symmetric positive semidefinite  matrix $A=(a_{ij})_{i,j}$ with $a_{ij}\in R$ the matrix $L(A):=(L(a_{ij}))_{i,j}$ is positive semidefinite as well. For technical reasons, we declare every map $L\colon R\to R$ to be $0$-positive. If $L$ is $k$-positive for every $k$, then, following the nomenclature of operator theory (see, e.g., \cite[\S1.5]{ozawa}), we say that $L$ is \emph{completely positive}.
\end{Def}

\begin{rem}
 Note that if $L\colon R \to R$ is $k$-positive, then it is also $k'$-positive for all {$k'\in\{0,\ldots, k\}$} since principal submatrices of positive {semidefinite} matrices are again positive semidefinite.
\end{rem}

\begin{lem}\label{lem:sdpdeg}
 Let $k\in\Z_{>0}$.
 A map $L\colon R\to R$ preserves all formulas in $\Gamma_k$ {(}respectively $\Gamma${)}
  if and only if $L$ is unital,  $\R$-linear, and $k$-positive {(}respectively completely positive{)}.
\end{lem}

\begin{proof}
 Assume that $L$ preserves all formulas in  $\Gamma_k$.  The formula $x=1$ is equivalent to $${\begin{pmatrix}x-1\end{pmatrix}\succeq0 \wedge \begin{pmatrix}1-x\end{pmatrix}\succeq0.}$$ Thus $L(1)=1$. A similar argument shows that $z=\lambda\cdot x+\mu\cdot y$ is in $\Gamma_k$ for all $\lambda,\mu\in\R$. 
 To prove that $L$ is $k$-positive, 
 note that a symmetric matrix $A = (a_{ij})_{1\leq i,j \leq k}$ is positive semidefinite if and only if 
 $(a_{11},a_{12},\dots,a_{1k},a_{22},a_{23},\dots,a_{{kk}})$  {satisfies} 
 the linear matrix inequality 
 $\sum_{1 \leq i\leq j\leq k} x_{ij} E_{ij} \succeq 0$  where $E_{ij}$ has entry $1$ at position $i,j$ and at position $j,i$ and entry $0$ at all other positions. Since $L$ preserves this {formula},
 we have that $L(A) \succeq 0$ whenever $A \succeq 0$. 
 
 Conversely, assume that $L$ is unital, $\R$-linear, and $k$-positive. Consider $k \times k$ real symmetric matrices $A_0,A_1, \ldots,A_n$.  
 Let $$K=\{x\in R^n\mid A_0+x_1A_1+\cdots+x_nA_n \succeq0\}$$ and let $x=(x_1,\ldots,x_n)\in K$. Then $L(x)=(L(x_1),\ldots,L(x_n))$ is also in $K$. Indeed, by the properties of $L$ we have \begin{align*}
 A_0+L(x_1)A_1+\cdots+L(x_n)A_n & = L(A_0 + x_1 A_1 + \cdots + x_n A_n ) \succeq 0. && \qedhere
 \end{align*}
\end{proof}

\begin{lem}\label{lem:matrixlinear}
 Let $R\subset R'$ a field extension and  $L\colon R' \to R'$ be an $R$-linear map. Let $A$ be an $m\times n$ matrix with entries in $R'$. Let $M$ and $N$ be $a\times m$ and $n\times b$ matrices with entries in $R$. Then we have $L(MAN)=ML(A)N$.
\end{lem}

\begin{proof}
 This follows directly from $R$-linearity.
\end{proof}

Now let $\Lambda$ a divisible ordered {abelian} group. We consider  $H=R[[\epsilon^\Lambda]]$,
the field of Hahn series 
over $R$ with value group $\Lambda$ and the valuation $v\colon H\to\Lambda\cup\{\infty\}$ given by $v \left (\sum_{e \in \Lambda} c_e \epsilon^e\right ) = \min\{ e \in \Lambda \mid c_e \neq 0\}$.
Since $H$ is real closed~\cite[page 218, (2)]{Alling}, it is equipped with an order $\leq$ defined by $x \leq y \Leftrightarrow \exists z (z^2 + x = y)$.
Further, let $\fo$ the valuation ring of $H$, namely all elements of $H$ with nonnegative valuation. Then $\fo$ is a local ring with residue field $R$, the maximal ideal consisting of all elements with positive valuation. We denote by $\pi\colon\fo\to R$ the natural projection.

\begin{lem}\label{lem:kfromkm1}
  Let $L\colon H\to H$ be $R$-linear and $(k-1)$-positive
  and suppose that for all $a_1,\ldots,a_k \in H_{>0}$ with $v(a_1)<\cdots<v(a_k)$ the matrix $L(a\cdot a^{\top})$, where $a=(a_1,\ldots,a_k)^{\top}$, is positive semidefinite. Then $L$ is $k$-positive.
\end{lem}

\begin{proof}
 Let $B$ be a $k\times k$ symmetric positive semidefinite matrix over $H$. Then $$B=\sum_{i=1}^k b_i\cdot b_i^\top$$for some column vectors $b_1,\ldots,b_k\in H^k$. Since $L$ is $R$-linear, we have $$L(B)=\sum_{j=1}^k L(b_i\cdot b_i^\top).$$ Since the sum of positive semidefinite matrices is positive semidefinite, it therefore suffices to show that $L(a\cdot a^{\top})$ is positive semidefinite for all $a=(a_1,\ldots,a_k)^{\top}$ with $a_1,\dots,a_k \in H$.
 Since $L$ is $(k-1)$-positive, we may  assume that $a_i\neq 0$ for all $i$. Let $M\in\GL_n(R)$. 
 {The matrix $L(a\cdot a^{\top})$ is positive semidefinite if and only if
 $M L(a \cdot a^{\top}) M^{\top}$ is positive semidefinite, which 
 equals
 $L((Ma)\cdot(Ma)^{\top})$ 
 by \Cref{lem:matrixlinear}.} Thus, after choosing $M$ to be a suitably signed permutation matrix, we may assume that $a_1\geq a_2\geq\cdots\geq a_k>0$. This implies that $v(a_1)\leq\cdots\leq v(a_k)$. Now, after choosing $M$ to be a {product of suitable elementary matrices}, we can make each inequality strict.
\end{proof}

\begin{lem}\label{lem:respsd}
 Let $A$ be a symmetric $n\times n$ matrix with entries in $\fo$. Then the following holds:
\begin{enumerate}
    \item If {$A = (a_{ij})_{1\leq i,j\leq n}$ is positive semidefinite, then $\pi(A) = (\pi(a_{ij}))_{1\leq i,j\leq n}$} is positive semidefinite.
    \item If $\pi(A)$ is positive definite, then $A$ is positive definite.
\end{enumerate}
\end{lem}

\begin{proof}
 We use the well-known fact that a symmetric matrix is positive semidefinite (positive definite) if and only if all its principal minors are nonnegative (positive, respectively); see, e.g.,~\cite[Theorem 7.2.5]{HornJohnson}).
 Let $x\in \fo$ with $x\geq0$. Since every nonnegative element in $\fo$ has a square root in $\fo$, {we have that
 $\pi(x) \in R$ has a square root as well and hence $\pi(x)\geq0$. If $B$ is a principal submatrix of $A$, then $\det(\pi(B)) = \pi(\det(B))$. 
 Therefore, if every principal minor of $A$ is nonnegative, the same is true for every principal minor of $\pi(A)$.} This shows (1). For proving (2) assume that $A$ is not positive definite. Then there is a {principal} minor of $A$ which is nonpositive. The same minor of $\pi(A)$ will also be nonpositive showing that $\pi(A)$ is not positive definite.
\end{proof}

\begin{rem}
Note that the converse of neither $(1)$ nor $(2)$ in \Cref{lem:respsd} is true, as witnessed by the $1 \times 1$ matrices $(-\epsilon)$ and $(\epsilon)$.
\end{rem}

For any function $f\colon\Lambda\to R$ we define the function $L_f\colon H\to H$ via $$L_f \left (\sum_{e\in\Lambda}c_e\epsilon^e \right )=\sum_{e\in\Lambda}f(e)c_e\epsilon^{e}.$$ We want to derive a criterion on $f$ for $L_f$ to be $k$-positive. We start with the following refinement of \Cref{lem:kfromkm1}.

\begin{lem}\label{lem:kfromkm2}
 Let $L\colon H\to H$ be $R$-linear and $(k-1)$-positive and suppose that for all $a_1,\ldots,a_k \in H_{>0}$ with $v(a_1)<\cdots<v(a_k)$ and $\pi(\frac{a_i}{\epsilon^{v(a_i)}})=1$ the matrix $L(a\cdot a^{\top})$ is positive semidefinite, where $a=(a_1,\ldots,a_k)^{\top}$. Then $L$ is $k$-positive.
\end{lem}

\begin{proof}
 This follows in the same way as \Cref{lem:kfromkm1} by another multiplication with a suitable diagonal matrix.
\end{proof}

\begin{cor}\label{cor:kposcrit}
Let $L\colon H\to H$ be $R$-linear and suppose that for all $a_1,\ldots,a_k \in H_{>0}$ with $v(a_1)<\cdots<v(a_k)$ and $\pi(\frac{a_i}{\epsilon^{v(a_i)}})=1$ the matrix $L(a\cdot a^{\top})$ is positive semidefinite, where $a=(a_1,\ldots,a_k)^{\top}$. Then $L$ is $k$-positive.
\end{cor}

\begin{proof}
For every {$i\in\{0,\ldots,k\}$} we prove that $L$ is $i$-positive by induction on $i$. The case $i=0$ is trivial. Let {$i\in\{1,\ldots,k\}$}  and assume that $L$ is $(i-1)$-positive. Let $a_1,\ldots,a_k \in H_{>0}$ with $v(a_1)<\cdots<v(a_k)$ and $\pi(\frac{a_i}{\epsilon^{v(a_i)}})=1$. By assumption $L(a\cdot a^{\top})$ is positive semidefinite. This implies in particular that the top left $i\times i$ submatrix of $L(a\cdot a^{\top})$ is positive semidefinite as well. Thus, we can apply \Cref{lem:kfromkm2} which proves that $L$ is $i$-positive. 
\end{proof}

For later reference we also note the following condition for $L_f$ being a bijection.

\begin{lem}\label{lem:lfbij}
 Let $f \colon \Lambda\to R$ be such that $L_f \colon H\to H$ is unital and $2$-positive. Then $L_f$ is bijective.
\end{lem}

\begin{proof}
 Let $a\in\Lambda$ and consider the positive semidefinite matrix $$A=\begin{pmatrix}
                                                                    \epsilon^a& 1\\ 1& \epsilon^{-a}
                                                                   \end{pmatrix}.$$
 Since $L_f$ is unital, we have $$L_f(A)=\begin{pmatrix}
                                                                    f(a)\epsilon^a& 1\\ 1& f(-a)\epsilon^{-a}
                                                                   \end{pmatrix}.$$
 Since $L_f$ is $2$-positive, this shows that $f(a)f(-a)\geq1$, so in particular $f(a)\neq0$.
 Now one can define $$g:\Lambda\to R,\, a\mapsto \frac{1}{f(a)}$$ and one has $L_f\circ L_g=L_g\circ L_f=\textrm{id}_H$.
\end{proof}

\begin{Def}\label{def:totpos}
 A function $f\colon\Lambda\to R$ is called \emph{$k$-positive semidefinite}
 if for all $x_1,\ldots,x_k\in\Lambda$ the matrix $(f(x_i+x_j))_{i,j}$ is positive semidefinite. It is called \emph{$k$-positive definite} if for all pairwise different $x_1,\ldots,x_k\in\Lambda$ the matrix $(f(x_i+x_j))_{i,j}$ is 
 positive definite. Finally, the function $f\colon\Lambda\to R$ is called \emph{positive semidefinite} (\emph{positive definite}\footnote{{In~\cite[page 88]{ressel}, positive definite functions are called \emph{positive definite in the semigroup sense}; we use the shorter form since no confusion can arise in our context.}}) if $f$ is $k$-positive semidefinite ($k$-positive definite) for all $k\in\Z_{\geq0}$.
\end{Def}

\begin{rem}\label{rem:poskconvex}
 If $f_1,f_2\colon\Lambda\to R$ are $k$-positive (semi-)definite and $\lambda,\mu\in R_{>0}$, then $\lambda f_1$, $\mu f_2$ and $\lambda f_1+\mu f_2$ are $k$-positive (semi-)definite as well. Indeed, this follows from the fact that the cone of positive (semi-)definite matrices is convex. 
\end{rem}

The motivation for considering \Cref{def:totpos} is the following result.

\begin{thm}\label{thm:totimpcomp}
 Let $f\colon\Lambda\to R$ and consider $L_f\colon H\to H$ as defined above.
 \begin{enumerate}
     \item If $f$ is $k$-positive definite, then $L_f$ is $k$-positive.
     \item If $L_f$ is $k$-positive, then $f$ is $k$-positive semidefinite.
 \end{enumerate}
\end{thm}

\begin{proof}
In order to prove $(2)$ let $v_1,\ldots,v_k\in\Lambda$.
We have to prove that $(f(v_i+v_j))_{i,j}$ is positive semidefinite. 
Let $a=(\epsilon^{v_1},\ldots,\epsilon^{v_k})^{\top}$. By assumption $L_f(a\cdot a^{\top})$ is positive semidefinite. Let $b=(\epsilon^{-v_1},\ldots,\epsilon^{-v_k})^{\top}$ and let $B=b\cdot b^{\top}$ which is positive semidefinite. 
Then the Hadamard product $C$ of $L_f(a\cdot a^{\top})$ with $B$ 
 is positive semidefinite as well
 by the Schur product theorem.
 Hence, \Cref{lem:respsd}(1) implies that $\pi(C)$ is positive semidefinite. 
 One calculates that $\pi(C)$ is the matrix $(f(v_i+v_j))_{1\leq i,j\leq k}$, which proves $(2)$. 

We will prove $(1)$ by applying \Cref{cor:kposcrit}. To that end let $a_1,\ldots,a_k \in H_{>0}$ be such that $v(a_1)<\cdots<v(a_k)$ and $\pi(\frac{a_i}{\epsilon^{v(a_i)}})=1$. We have to show that the matrix $L_f(a\cdot a^{\top})$ is positive semidefinite, where $a=(a_1,\ldots,a_k)^{\top}$. Let $b=(\epsilon^{-v(a_1)},\ldots,\epsilon^{-v(a_k)})^{\top}$ and $B=b\cdot b^{\top}$. Since $B$ is positive semidefinite of rank one with nonzero entries, the matrix $L_f(a\cdot a^{\top})$ is positive semidefinite if and only if its Hadamard product $C$ with $B$ is positive semidefinite. As above one calculates that $\pi(C)$ is the matrix $(f(v(a_i)+v(a_j)))_{1\leq i,j\leq k}$ which is positive definite since $f$ is $k$-positive definite and $v(a_1),\dots,v(a_k)$ are pairwise different. Thus, $C$ is positive definite by \Cref{lem:respsd}(2). Therefore, $L_f(a \cdot a^{\top})$ is positive semidefinite and $(1)$ follows from~\Cref{cor:kposcrit}.
\end{proof}

\begin{rem}\label{rem:groupsos}
 Another point of view on positive semidefinite functions can be taken by looking at the group ring $R[\Lambda]$. The group ring $R[\Lambda]$ can be identified with the subring of $H=R[[\epsilon^{\Lambda}]]$ consisting of series with finite support. Now any function $f\colon\Lambda \to R$ can be extended linearly to an $R$-linear functional $T_f\colon R[\Lambda]\to R$. An element of $R[\Lambda]$ is of the form $a=\sum_{i=1}^k c_i \epsilon^{a_i}$ for $k \in {\mathbb N}$, $(c_1,\dots,c_k)\in R^k$, and $(a_1,\dots,a_k) \in \Lambda^k$. Then we have \begin{align*}
 T_f(a^2) &=T_f \left (\sum_{i,j = 1}^k c_i c_j \epsilon^{a_i + a_j} \right) = \sum_{i,j}^k c_ic_j f(a_i+a_j)\\& = \begin{pmatrix}c_1&\hdots &c_k
 \end{pmatrix}\cdot (f(a_i+a_j))_{1\leq i,j\leq k}\cdot \begin{pmatrix}c_1\\\vdots \\c_k
 \end{pmatrix}
 .    
 \end{align*}
 From this we see
 that $f$ is positive semidefinite (positive definite) if and only if $T_f(a^2)\geq0$ ($T_f(a^2)>0$, respectively) for all $a\in R[\Lambda] \setminus \{0\}$. 
\end{rem}

For the rest of this section we focus on the case $\Lambda=\R$.
Let $\cH(\C)$ be the ring of holomorphic functions on $\C$.
We have a homomorphism of $\R$-algebras
$\Psi \colon \R[\Lambda]\to \cH(\C)$ defined as follows. 
For $a\in\R$, set $\Psi(\epsilon^a)$
to be the holomorphic function $z\mapsto\exp(az)$ and extend linearly to all of $\R[\Lambda]$. 

\begin{lem}\label{lem:identhm}
Let  $g\in\R[\Lambda]$. If $\Psi(g)$ vanishes on $[a,b]$ for $a<b$, then $g=0$ as element of $\R[\Lambda]$. In particular, the homomorphism $\Psi$ is injective.
\end{lem}

\begin{proof}
Let $g=\sum_{i=1}^nc_i\epsilon^{a_i}$ for some $c_i,a_i\in\R$, $i=1,\ldots,n$, with $a_1<\cdots<a_n$.
 The identity theorem from complex analysis implies that $\Psi(g)$ vanishes on all of $\C$. Thus we have $\sum_{i=1}^nc_i\exp(a_iz)=0$ for all $z\in\C$. Linear independence of characters now implies that $c_i=0$ for {$i \in \{1,\ldots,n\}$} because the functions $\R\to\R,\, t\mapsto\exp(a_it)$ for {$i \in \{1,\ldots,n\}$} are pairwise different characters.
\end{proof}

This enables us to construct positive definite functions $\R\to\R$.

\begin{ex}\label{ex:strictpos}
 Let $\Lambda=\R$ and $\mu$ a measure on $\R$ with $\int \exp(ax)d\mu(x)<\infty$ for all $a\in\R$. Then the functional $$T\colon\R[\Lambda]\to\R,\, g\mapsto\int \Psi(g)(x)d\mu(x)$$ clearly satisfies $T(g^2)\geq0$ for all $g\in\R[\Lambda]$. If we choose for instance $\mu$ to be {the pushforward to $\R$ of the Lebesgue measure} on some compact interval $[a,b]$ with $a<b$, then we even have $T(g^2)>0$ for every $g\in\R[\Lambda] \setminus \{0\}$. Indeed, for any $g\in\R[\Lambda]$ with $T(g^2)=0$ the function $\Psi(g)$ must vanish on $[a,b]$,  which implies $g=0$ by \Cref{lem:identhm}. 
 Let $f\colon\R\to\R$ defined by $f(a)=T(\epsilon^a)$ for $a\in\R$. Since $T$ is $\R$-linear, we have $T=T_f$ and thus $f$ is positive definite by \Cref{rem:groupsos}.
\end{ex}

\begin{cor}\label{cor:exposmap}
Let $\Lambda$ be a divisible ordered abelian group whose dimension as a $\Q$-vector space is at most $2^{\aleph_0}$. Further, let $R$ be a real closed field extension of $\R$. Then there exists an $R$-linear map $T \colon R[\Lambda]\to R$ which satisfies $T(g^2)>0$ for every $g \in R[\Lambda] \setminus \{0\}$.
\end{cor}

\begin{proof}
By \Cref{ex:strictpos} there exists a  positive definite map $f\colon\R\to\R$. 
 By our assumptions on $\Lambda$ there is an injective group homomorphism $\iota_1\colon\Lambda \to\R$. Let $\iota_2 \colon \R\to R$ be the inclusion map. Then the map $\bar{f}=\iota_2\circ f\circ\iota_1\colon\Lambda \to R$ is also positive definite. 
 By \Cref{rem:groupsos} the map $T_{\bar{f}}\colon R[\Lambda]\to R$ has the desired properties.
\end{proof}

\section{The group ring $R[\Q^n]$}\label{sect:groupring}
From now on let $R$ be a real closed field (not necessarily an extension of $\R$). {We fix a total order on $\Q^n$ that is translation-invariant, for instance the lexicographical ordering.} As in \Cref{rem:groupsos} we consider the group ring $R[\Q^n]$ as the subring of $R[[\epsilon^{\Q^n}]]$ of series with finite support. For $a=(a_1,\ldots,a_n)\in\Q^n$ we write $$\epsilon^{a}=\epsilon_1^{a_1}\cdots\epsilon_n^{a_n}.$$ By plugging in positive elements from $R$ for $\epsilon_i$ we can regard elements from $R[\Q^n]$ as functions on the positive orthant $R_{>0}^n$. 
Clearly, every element of $R[\Q^n]$ that can be written as a sum of squares of elements of $R[\Q^n]$ is nonnegative on $R_{>0}^n$. We will see in \Cref{cor:notpsdsos} that the converse to this is not true. We will consider the polynomial ring $R[\underline{x}]=R[x_1,\ldots,x_n]$ as subring of $R[\Q^n]$ by replacing $x_i$ by $\epsilon_i$.

\begin{rem}
 Evaluating elements $f\in R[\Q^n]$ at negative numbers is in general not well defined since the exponents can have even denominator. Similarly, a polynomial can be a sum of squares of elements from the group ring $R[\Q^n]$ without being nonnegative on all of $R^n$. {For example, $\epsilon \in R[\epsilon]$ takes negative values on $R$, but equals  $(\epsilon^{\frac{1}{2}})^2$.}
\end{rem}

\begin{rem}\label{rem:sos}
 A polynomial $p\in R[\underline{x}]$ is a sum of squares of elements in $R[\Q^n]$ if and only if $p(x_1^m,\dots,x_n^m)$ is a sum of squares of polynomials {from $R[\underline{x}]$} for some $m\in\Z_{>0}$. Indeed, negative exponents cannot appear in the sum of squares decomposition since lowest degree terms cannot cancel out.
\end{rem}

\begin{ex}
 Consider the Motzkin polynomial $$M=x_3^6-3x_1^2x_2^2x_3^2+x_1^2x_2^4+x_1^4x_2^2\in R[x_1,x_2,x_3].$$ It follows from the inequality of arithmetic and geometric means that $M$ is nonnegative on $R^3$
{(since $\frac{1}{3} x_1^4x_2^2 + \frac{1}{3} x_1^2 x_2^4 + \frac{1}{3} x_3^6 \geq x_1^{4/3}x_2^{2/3}x_1^{2/3}x_2^{4/3}x_3^{6/3}$)}
 but it is not a sum of squares of polynomials from $R[\underline{x}]$ \cite{motzkin}. However, $M(x_1^2,x_2^2,x_3^2)$ is a sum of squares of polynomials \cite[Example~5.3]{MR985241}:
 \[M(x_1^2,x_2^2,x_3^2)=
 2\cdot(x_1^3x_2^3-x_1x_2x_3^4)^2+(x_1^4x_2^2-x_1^2x_2^4)^2+(x_3^6-x_1^2x_2^2x_3^2)^2 .
 \]
 This shows  that $M$ is a sum of squares of elements of $R[\Q^3]$.
\end{ex}

\begin{ex}\label{cor:notpsdsos}
 Consider the \emph{Horn polynomial} $$h=(x_1+x_2+x_3+x_4+x_5)^2-4(x_1x_2+x_2x_3+x_3x_4+x_4x_5+x_5x_1).$$It was proven in \cite[p.~335]{hornex} that $h$ is nonnegative on the positive orthant. However, for any integer $k>0$ the polynomial $h(x_1^k,\ldots,x_5^k)$ is not a sum of squares of elements of $R[{\underline x}]$. This was shown in \cite[\S4]{hornrez} by adapting the proof from \cite{hornex} for the case $k=2$. 
 This shows  that $h$ is not a sum of squares of elements of $R[\Q^5]$.
 
 For the sake of completeness, we repeat the argument here. Assume $h$ is a sum of squares in $R[\Q^5]$. Then for some integer $d>0$ the polynomial $$H_d=h(x_1^{{d}},\ldots,x_5^{{d}})=g_1^2+ \dots +g_m^2$$ is a sum of squares of homogeneous polynomials $g_j$ of degree $d$ (since highest and lowest degree terms cannot cancel out). Restricting $H_d$ to $x_{i}=x_{i+1}=0$ we obtain $$(x_{i+2}^d-x_{i+3}^d+x_{i+4}^d)^2$$ {for $i \in \{ 1,\dots,5\}$} where the indices are taken modulo $5$. Thus, {for $j \in \{1,\dots,m\}$} the restriction of $g_j$ to $x_{i}=x_{i+1}=0$ must be a scalar multiple of $x_{i+2}^d-x_{i+3}^d+x_{i+4}^d$. 
 {Indeed, if $$(x_{i+2}^d-x_{i+3}^d+x_{i+4}^d)^2=g_1'^2+\dots+g_m'^2$$for $g_j'\in\R[x_{i+2},x_{i+3},x_{i+4}]$, then each $g_j'$ must vanish on the real zero set of $x_{i+2}^d-x_{i+3}^d+x_{i+4}^d$. The latter is Zariski dense in the complex zero set which implies by Hilbert's Nullstellensatz that $g_j'$ is divisible by the irreducible polynomial $x_{i+2}^d-x_{i+3}^d+x_{i+4}^d$.}
 Letting $c_i$ be the coefficient of $x_{i}^d$ in $g_j$, this shows that $c_i=-c_{i+1}$. Since $5$ is an odd number, this actually implies that $c_1=\cdots=c_5=0$. In particular, we have that $g_j(1,0,\ldots,0)=0$ {for $j \in \{1,\ldots,m\}$}. However, since $H_d(1,0,\ldots,0)\neq0$, this is a contradiction.
\end{ex}

\begin{ex}
 The Horn polynomial considered in \Cref{cor:notpsdsos} is a special case of the following more general construction. Recall that a \emph{stable set} of a graph $G=(V,E)$ is a subset of elements of $V$ that are pairwise not adjacent to each other. The \emph{stability number} {$\alpha(G)$} of $G$ is the largest size of a stable set of $G$. By a theorem of Motzkin--Straus \cite{motzkinstraus} the quadratic form associated to the symmetric matrix $$Q(G)=\alpha(G)(I+A)-J$$ is nonnegative on the positive orthant \cite[Example~3.163]{parrilo}, i.e., the matrix $Q(G)$ is \emph{copositive}. Here, $A$ is the adjacency matrix of $G$ and $J$ is the all-ones matrix. The Horn polynomial corresponds to the matrix $Q(C_5)$ where $C_5$ is the five-cycle. The argument from \Cref{cor:notpsdsos} can in fact be applied to any odd cycle of size at least five.
\end{ex}


We also note that nonnegative elements in $\R[\Q^n]$ cannot be distinguished from sums of squares by an {$\R$-}linear functional $\R[\Q^n]\to\R$.

\begin{prop}
 Let $T\colon\R[\Q^n]\to\R$ be an $\R$-linear functional with $T(a^2)\geq0$ for all $a\in\R[\Q^n]$ and let $p\in\R[\Q^n]$ be nonnegative on $R_{>0}^n$. Then $T(p)\geq0$.
\end{prop}

\begin{proof}
In the case $n=1$ the statement of \cite[Proposition~5.10, page 211]{ressel} says that for every such $T$ there is a measure $\mu$ on $[0,\infty)$ such that $T(g)=\int g d\mu$. This shows in particular that $T(p)\geq0$.
 The case $n>1$ then follows from \cite[Theorem~5.4, page 204]{ressel}.
\end{proof}

However, if we replace the target field by some real closed extension field, then separation by a linear functional is always possible. The following theorem follows from the results of \cite{rcsep}. We will give self-contained proofs in \Cref{sec:appendi}.

\begin{thm}\label{thm:seplin}
 Let $p\in R[\Q^n]$ not be a sum of squares of elements of $R[\Q^n]$. Then there exists a real closed field extension $R'\supset R$ and an $R'$-linear functional $T\colon R'[\Q^n]\to R'$ such that for every $a\in R'[\Q^n] \setminus \{0\}$ we have 
 $T(a^2)>0$  and $T(p)<0$.
\end{thm}

\begin{proof}
 By \Cref{cor:mainsep} there exists a real closed field extension $R'\supset R$ and an $R'$-linear functional $T'\colon R'[\Q^n]\to R'$ such that  for every $a\in R'[\Q^n]$ we have $T'(a^2)\geq0$ and $T'(p)<0$. Further, by \Cref{cor:exposmap} there exists an $R'$-linear map $T_0 \colon R'[\Q^n]\to R'$ that satisfies $T_0(a^2)>0$ for all $0\neq a\in R'[\Q^n]$. Let $\lambda\in R'$ be such that $\lambda>0$ and $\lambda |T_0(p)|<|T'(p)|$. Then $T=\lambda T_0+T'$ has the desired properties.
\end{proof}

Now we are ready to formulate and prove our main result.

\begin{Def}
 Let $S\subset\Z_{\geq0}^n$ be a finite set and $R$ a real closed field. We denote by $\Pol(S,R)$ the (finite dimensional) $R$-vector space of all polynomials in $n$ variables and coefficients in $R$ whose support is contained in $S$. We further denote by $\Pol_+(S,R)$ the convex cone of all $p\in\Pol(S,R)$ such that $p(x)\geq0$ for all $x\in R^n$. Finally, we denote by $\Sigma(S,R)$ the convex cone of all $P\in\Pol(S,R)$ that are a sum of squares of polynomials (with coefficients in $R$). If we do not specify the real closed field, then we work over the real numbers, i.e., we write $\Pol(S)=\Pol(S,\R)$, $\Pol_+(S)=\Pol_+(S,\R)$, and $\Sigma(S)=\Sigma(S,\R)$.
\end{Def}

\begin{rem}\label{rem:overz}
 Note that $\Pol_+(S,R)$ is the set of all tuples $(c_s)_{s\in S}\subset R^S$ such that $$\forall x_1 \cdots\forall x_n\colon\,\sum_{s\in S}c_s\cdot\prod_{i=1}^nx_i^{s_i}\geq0$$holds true over $R$. This is a formula in the language of real closed fields without parameters. 
 Such a formula also exists for $\Sigma(S,R)$. {Indeed, let $T_1=\Z_{\geq0}^n\cap\conv(S)$ and $T_2=\Z_{\geq0}^n\cap(\frac{1}{2}\conv(S))$. Every element of $\Sigma(S,R)$ is a sum of squares of elements of $\Pol(T_2,R)$, i.e., lies in the conical hull of the set $$\{q^2\mid q\in\Pol(T_2,R)\}\subset\Pol(T_1,R).$$
 Then by Carath\'eodory's theorem for the conical hull every such element can be written as a sum of $|T_1|=\dim(\Pol(T_1,R))$ many squares.
 Therefore, we have that $\Sigma(S,R)$ is the set of all $p\in\Pol(S,R)$ for which there exist $c_{i,t}\in R$, for $i \in \{1,\dots,|T_1|\}$
 and $t\in T_2$, such that $$p=\sum_{i=1}^{|T_1|}\left(\sum_{t\in T_2} c_{i,t} \prod_{j=1}^nx_j^{t_j}\right)^2.$$} This shows that $\Sigma(S,R)$ is defined by a formula in the language of real closed fields 
 without parameters. 
\end{rem}

\begin{rem}
 For all finite sets $S$ the set $\Sigma(S)$ is a spectrahedral shadow \cite[\S6.3.1, page 263]{MR3050245}.
\end{rem}

\begin{prop}\label{thm:monomialshadows}
 Let $S\subset\Z_{\geq0}^n$ be a finite set. Assume that there exists a real closed field $R$ and $p\in\Pol_+(S,R)$ such that $p$ is not a sum of squares of elements from $R[\Q^n]$. Then $\Pol_+(S)$ is not a spectrahedral shadow.
\end{prop}

\begin{proof}
 Without loss of generality we may assume that $R$ contains $\R$.
 By \Cref{thm:seplin} there exists a real closed field extension $R'\supset R$ and an $R'$-linear functional $T\colon R'[\Q^n]\to R'$ such that $T(a^2)>0$ for all $a\in R'[\Q^n] \setminus \{0\}$ and $T(p)<0$. {Let $f\colon\Q^n\to R'$ defined by $f(w)=T(\epsilon_1^{w_1}\cdots\epsilon_n^{w_n})\in R'$ for $w\in\Q^n$. By the $R'$-linearity of $T$ we then have $T = T_f$ as in \Cref{rem:groupsos}}. As $T(a^2) > 0$ we have that $f$ is positive definite (see \Cref{rem:groupsos}). After replacing $T$ by $\frac{1}{T(1)}\cdot T$, we can further assume that $f(0)=1$. Therefore, we have that $L_f\colon R'[[\epsilon^{\Q^n}]]\to R'[[\epsilon^{\Q^n}]]$ is completely positive by \Cref{thm:totimpcomp}, $\R$-linear by construction, and unital because $f(0)=1$. The polynomial $q=p(\epsilon_1x_1,\ldots,\epsilon_nx_n)$ with coefficients in $R[[\epsilon^{\Q^n}]]$
 is nonnegative on $R'[[\epsilon^{\Q^n}]]^n$. By construction we have $$L_f(q)(\epsilon_1^{-1},\ldots,\epsilon_n^{-1})=T_f(p)<0.$$ Thus, $L_f$ preserves all formulas in $\Gamma$ by \Cref{lem:sdpdeg} but not the formula defining $\Pol_+(S,R'[[\epsilon^{\Q^n}]])$.
  This shows that $\Pol_+(S)$ is not a spectrahedral shadow.
\end{proof}

A first-order formula is called \emph{positive} 
if it is built from atomic formulas by existential and universal quantification, conjunction, and disjunction.
That is, only negation from first-order logic is forbidden.

\begin{prop}\label{prop:pos-logic}
{Let $S \subset {\mathbb Z}^n_{\geq 0}$ and $p \in \Pol_+(S,R)$ be as in \Cref{thm:monomialshadows}. Then
$\Pol_+(S)$ does not have a positive definition over the relational structure whose relations consist of all spectrahedra and the inequality relation $\neq$.}  
\end{prop}
\begin{proof}
Note that by \Cref{lem:lfbij}
the map $L_f \colon R'[[\epsilon^{\Q^n}]]\to R'[[\epsilon^{\Q^n}]]$ constructed in the proof of \Cref{thm:monomialshadows} 
is surjective. It follows from the easy direction of the Lyndon preservation theorem for positive first-order logic (see, e.g., Corollary 10.3.5 in~\cite{HodgesLong}) and the Tarski principle for real-closed fields (similarly as outlined for existential positive logic in the introduction) that
the convex semi-algebraic set $\Pol_+(S)$ does not have a definition by a positive formula over spectrahedra. 
In fact, the map $L_f$ is also injective and hence preserves $\neq$, and so $\Pol_+(S)$ is inexpressible even if we additionally allow the use of $\neq$.
\end{proof}

\begin{rem} 
If full first-order logic is permitted, then the graph of multiplication can already be defined from fairly restricted sets of relations over the reals~\cite{Peterzil,MarkerPeterzilPillay} (and in combination with the graph of addition we then obtain all semialgebraic relations). In particular, it is easy to see that all semialgebraic relations can be defined with first-order formulas and parameters from the graph of addition and the spectrahedron defined by $y \geq x^2$.
\end{rem}

Let $R$ be a real closed field, $d>0$ an integer, $S\subset\Z_{\geq0}^n$ a finite set and $s\in\{\pm1\}^n$. 
We consider the following isomorphisms of $R$-vector spaces:
\begin{eqnarray*}
 \psi_d\colon& \Pol(S,R)\to\Pol(d\cdot S,R),& p(x_1,\ldots,x_n)\mapsto p(x_1^d,\ldots,x_n^d),\\
 \varphi_s\colon& \Pol(S,R)\to\Pol(S,R),& p(x_1,\ldots,x_n)\mapsto p(s_1x_1,\ldots,s_n x_n).
\end{eqnarray*}
Further, we define $$\Sigma_d(S,R)=\bigcap_{s\in\{\pm1\}^n}\varphi_s(\psi_d^{-1}(\Sigma(d\cdot S,R))).$$
Again we write $\Sigma_d(S)=\Sigma_d(S,\R)$.
\begin{lem}\label{lem:props}
 The following hold true for all integers $d>0$:
 \begin{enumerate}
  \item $\Sigma_d(S,R)$ can be defined by a formula in the language of real closed fields without parameters. 
  \item $\Sigma_d(S,R)\subseteq \Sigma_{i\cdot d}(S,R)$ for all integers $i>0$.
  \item $\Sigma_d(S,R)\subseteq\Pol_+(S,R)$.
  \item $\Sigma_d(S)$ is a spectrahedral shadow.
 \end{enumerate}
\end{lem}

\begin{proof}
 Part $(1)$ follows from \Cref{rem:overz} and the fact that $\psi_d$ and $\varphi_s$ {are first-order definable without parameters in the language of real closed fields}. If a polynomial $p$ is a sum of squares, then $p(x_1^i,\ldots,x_n^i)$ is also a sum of squares. This shows $(2)$. Let $p\in\Sigma_d(S,R)$ and $x_1,\ldots,x_n\in R$. To show $(3)$, we have to show that $p(x_1,\ldots,x_n)\geq0$. For $i \in \{1,\dots,n\}$, let $s_i\in\{\pm1\}$ be such that $s_i x_i\geq0$. Then we let $y_i=\sqrt[d]{s_i x_i}$. We have $q=\psi_d(\varphi_{s}^{{-1}}(p))\in \Sigma(d\cdot S,R)\subseteq\Pol_+(d\cdot S,R)$. Therefore, $$0\leq q(y_1,\ldots,y_n)=(\varphi_{s}^{{-1}}(p))(s_1x_1,\ldots,s_nx_n)=p(x_1,\ldots,x_n).$$This shows $(3)$. We finally have that $\Sigma_d(S)$ is a spectrahedral shadow as the intersection of finitely many spectrahedral shadows.
\end{proof}

\begin{thm}\label{thm:main}
 Let $S\subset\Z_{\geq0}^n$ be a finite set. The following are equivalent:
 \begin{enumerate}
  \item $\Pol_+(S)$ is a spectrahedral shadow.
  \item For every real closed field $R$ and every polynomial $p\in\Pol_+(S,R)$, there exists $d>0$ such that $p(x_1^d,\ldots,x_n^d)$ is a sum of squares.
  \item There exists $d>0$ such that for every real closed field $R$ and every polynomial $p\in\Pol_+(S,R)$ the polynomial $p(x_1^d,\ldots,x_n^d)$ is a sum of squares.
  \item There exists $d>0$ such that for every polynomial $p\in\Pol_+(S)$ the polynomial $p(x_1^d,\ldots,x_n^d)$ is a sum of squares.
  \item $\Pol_+(S)$ has a positive definition over the relational structure whose relations consist of all spectrahedra and the inequality relation $\neq$.
 \end{enumerate}
\end{thm}

\begin{proof}
 The implication $(1)\Rightarrow(2)$ follows from \Cref{thm:monomialshadows} {and \Cref{rem:sos}}.
 Now assume $(2)$. We first claim that $$\Pol_+(S,R)=\bigcup_{k=1}^\infty\Sigma_k(S,R).$$ The inclusion $\supseteq$ is \Cref{lem:props}(3). Let $p\in\Pol_+(S,R)$. By assumption, for all $s\in\{\pm1\}^n$ there is $k_{s}>0$ such that $\psi_{k_{s}}(\varphi_{s}(p))\in\Sigma(k_{s}\cdot S,R)$. Therefore, letting $k=\prod_{s\in\{\pm1\}^n} k_{s}$ we have $\psi_{k}(\varphi_{s}(p))\in\Sigma(k\cdot S,R)$ and thus $p\in\Sigma_k(S,R)$.

 Both $\Pol_+(S,R)$ and each $\Sigma_k(S,R)$ can be defined by a formula without parameters. Thus, letting $\R^*$ be an  $\aleph_1$-saturation of $\R$, we have $\Pol_+(S,\R^*)=\Sigma_d(S,\R^*)$ for some $d>0$ by \cite[Theorem~2.2.11]{PosPols}. By the Tarski principle we thus have $\Pol_+(S,R)=\Sigma_d(S,R)$ for every real closed field $R$. This shows $(3)$.
 
 The direction $(3)\Rightarrow(4)$ is trivial. Finally, assume $(4)$ which implies $\Pol_+(S)=\Sigma_d(S)$. This shows $(1)$ by \Cref{lem:props}(4) and we have that $(1)$---$(4)$ are equivalent. 
 
 It is clear that $(1)$ implies $(5)$ and \Cref{prop:pos-logic} shows that $(5)$ implies $(2)$.
\end{proof}

\begin{rem}\label{rem:copos}
 \Cref{thm:main} also gives a characterization of when the set of all $p\in\Pol(S)$ which are nonnegative on the nonnegative orthant is a spectrahedral shadow. Indeed, a polynomial $p\in\Pol(S)$ is nonnegative on the nonnegative orthant if and only if $p(x_1^2,\ldots,x_n^2)\in\Pol_+(2S)$.
\end{rem}

\begin{rem}
 A closed convex cone $K\subseteq\R^n$ is a spectrahedral shadow if and only if its dual cone is a spectrahedral shadow (Exercise~6.23 in~\cite{MR3050245}). The dual cone of $\Pol_+(S)$ is the closed conical hull of the image of $\R^n$ under the monomial map given by the monomials from $S$.
\end{rem}

The following answers Question 5.2 in \cite{heltonniefalse}. 
For $s \in {\mathbb Z}^n_{\geq 0}$, we write $|s|$ for $s_1+\dots+s_n$.

\begin{cor}\label{cor:copos}
 The set of copositive matrices of size $n\geq5$ is not a spectrahedral shadow. 
 Moreover, it does not have a positive definition in the relational structure over the reals whose relations consist of all spectrahedra and the inequality relation.
\end{cor}

\begin{proof}
 By \Cref{rem:copos} the set of copositive matrices of size $n$ can be identified with the set of all polynomials $q\in\R[x_1,\ldots,x_n]$ that are homogeneous of degree two such that $q(x_1^2,\ldots,x_n^2)$ is nonnegative on $\R^n$. Thus, the set of copositive matrices of size $n$ is identified with
 $\Pol_+(2S)$ where $S=\{s\in\Z_{\geq0}^n\mid|s|=2\}$. We have seen in  \Cref{cor:notpsdsos} that the Horn polynomial $h$ satisfies $p=h(x_1^2,\ldots,x_n^2)\in\Pol_+(2S)$ but $p(x_1^d,\ldots,x_n^d)$ is not a sum of squares for every $d>0$. Thus \Cref{thm:main} implies the claim.
\end{proof}

For the rest of this section we summarize the current knowledge on whether $\Pol_+(S)$ is a spectrahedral shadow for a given finite set $S\subset\Z_{\geq0}^n$.

\begin{ex}[Hilbert]
 Let $S\subset\Z_{\geq0}^n$ a finite set and $\Delta_{\leq d}$ the set of all $s\in\Z_{\geq0}^n$ such that $|s|\leq d$.
 Hilbert \cite{hilbert1888} proved that every element of $\Pol_+(S)$ is a sum of squares if we are in one of the following cases:
 \begin{enumerate}
  \item $S\subseteq\Delta_{\leq 2}$;
  \item $n=1$;
  \item $n=2$ and $S\subseteq\Delta_{\leq 4}$.
 \end{enumerate}
 Thus, in all these cases $\Pol_+(S)$ is a spectrahedral shadow.
\end{ex}

\begin{ex}[Fawzi, Scheiderer]
 Let $S\subset\Z_{\geq0}^n$ a finite set which is \emph{downward closed} in the sense that if $s'\leq s$ (coordinate-wise) and $s\in S$, then $s'\in S$. Then by \cite[Theorem~3]{fawzi2019set} the set $\Pol_+(S)$ is a spectrahedral shadow if and only if $\Pol_+(S)=\Sigma(S)$. Thus, for every nonnegative polynomial which is not a sum of squares, letting $S$ be the downward closure of its support, we obtain an instance where $\Pol_+(S)$ is not a spectrahedral shadow. 
 Scheiderer \cite{heltonniefalse} used a related criterion in several examples. A generalization of \cite[Theorem~3]{fawzi2019set} to finite dimensional spaces of polynomials not necessarily spanned by monomials is given in \cite[Theorem~2.14]{bettiol2019convex}.
\end{ex}

\begin{ex}[SONC]\label{ex:rez}
 A finite subset $A\subset\Z_{\geq0}^n$ is called \emph{circuit set} if its convex hull $\conv(A)\subset\R^n$ is a $k$-dimensional simplex whose vertices $v_0,\ldots,v_k$ are in $2\Z_{\geq0}^n$ and $A\setminus\{v_0,\ldots,v_k\}\subseteq\{w\}$ for some $w$ in the relative interior of $\conv(A)$. For any finite $S\subset\Z_{\geq0}^n$ the cone of \emph{sums of nonnegative circuit polynomials} is defined as $$\textrm{SONC}(S)=\sum_{A\subseteq S \textrm{ a circuit set}}\Pol_+(A).$$ Averkov \cite[\S6]{Ave} showed that $\textrm{SONC}(S)$ has a semidefinite extended representation with LMIs of size at most two. Without the bound on the LMI sizes this also follows from a result by Reznick \cite{MR985241, hornrez} which implies that for every $p\in\textrm{SONC}(S)$ the polynomial $p(x_1^d,\ldots,x_n^d)$ is a sum of squares for $d=\max(2,n-1)$. In particular, we have that $\Pol_+(S)$ is a spectrahedral shadow whenever $\Pol_+(S)=\textrm{SONC}(S)$. By definition this is the case if $S$ is a circuit set. Equality is also straightforward to see if every element of $S$ is a vertex of $\conv(S)$. More cases of equality can be deduced from \cite[Theorem~11]{sage} and \cite[Theorem~8.1]{forsgaard2019algebraic}.
\end{ex}

The proof of \Cref{cor:copos} shows the existence of a set $S \subset {\mathbb Z}^n_{\geq 0}$ of size $\frac{1}{2}n(n+1)$, for $n \geq 5$, such that $\Pol_+(S)$ is not a spectrahedral shadow. For $n=5$ this means $|S|=15$. Moreover, Scheiderer~\cite[Remark~4.21]{heltonniefalse} constructs such an $S\subset\Z^2_{\geq0}$ with $|S|=14$.
We do not know the smallest size of a set $S\subset\Z_{\geq0}^n$ such that $\Pol_+(S)$ is {not} a spectrahedral shadow. The following shows that any such set must satisfy $|S|>5$.

\begin{cor}
 Let $S\subset\Z_{\geq0}^n$ with $|S|\leq5$. Then $\Pol_+(S)$ is a spectrahedral shadow.
\end{cor}

\begin{proof}
 Consider the convex hull $K=\conv(S)\subset\R^n$ and let $N$ be the number of vertices of $K$. First suppose that $N\leq2$. Then there is a finite subset $T\subset\Z_{\geq0}$ and $v\in\Z_{\geq0}^n$, $w\in\Z^n \setminus \{0\}$ such that $S=\{v+k\cdot w\mid k\in T\}$ and such that the entries of $w$ do not have a nontrivial common divisor.  If $\Pol_+(S)\neq\{0\}$ we can further assume that $v\in2\Z_{\geq0}^n$. Then the vector space isomorphism $$\Pol(T)\to\Pol(S),\, p\mapsto \underline{x}^v\cdot p(\underline{x}^w)$$identifies $\Pol_+(T)$ with $\Pol_+(S)$. Since every univariate nonnegative polynomial is a sum of squares, the set $\Pol_+(T)$, and hence $\Pol_+(S)$, is a spectrahedral shadow. 
 
In the case $N\geq3$ we claim that $\Pol_+(S)=\textrm{SONC}(S)$. This implies the claim since $\textrm{SONC}(S)$ is a spectrahedral shadow by \cite[\S6]{Ave}, see also \Cref{ex:rez}. First assume that $K$ is a simplex. Since $|S|\leq 5$, there are at most two elements of $S$ that are not vertices of $K$. Hence $\Pol_+(S)=\textrm{SONC}(S)$ follows from \cite[Theorem~11(1)]{sage} (applied to the trivial partition of $S$). Now assume that $K$ is not a simplex. We first note that $N\geq3$ implies that $\dim(K)\geq2$. Therefore, since $K$ is not a simplex, we even have $N\geq4$. Thus there is at most one element of $S$ that is not a vertex of $K$ and we can apply \cite[Theorem~11(2)]{sage} (again to the trivial partition).
\end{proof}

\section{Separating $\Gamma_k$ from $\Gamma_{k+1}$}\label{sec:ave}
The following criteria for a function $f\colon\R\to\R$ to be $k$-positive definite or semidefinite are due to Karlin~\cite{Karlin}.

\begin{lem}\label{lem:hankelcrit}
For $k \in {\mathbb Z}_{>0}$ let $f\in C^{2k-2}(\R)$ and for all $x\in\R$ consider the Hankel matrix $H_f(x)=(f^{(i+j-2)}(x))_{1\leq i,j\leq k}$ {of order $k$}.
\begin{enumerate}
    \item If $f$ is $k$-positive semidefinite, then $H_f(x)$ is positive semidefinite for every $x\in\R$.
    \item If $H_f(x)$ is positive definite for every $x\in\R$, then $f$ is $k$-positive definite. 
\end{enumerate}
\end{lem}

\begin{proof}
 Part (1) follows from applying suitable transformations to the matrix $$(f(x_i+x_j))_{i,j}$$ for $x_1<\cdots<x_k$ and then letting all $x_i$ converge to the same value $x$ as it is done in \cite[\S2.1]{Karlin}.
 Part (2) follows from \cite[Thm.~2.6]{Karlin} applied to the kernel $K(x,y)=f(x+y)$.
\end{proof}

This criterion allows us to exhibit a function $\R\to\R$ that is $k$-positive definite, but not $k+1$-positive semidefinite. First we make some observations on the Hankel matrices $H_f(x)$.

\begin{rem}
 By definition the Hankel matrix $H_f(x)$ is linear in $f$.
\end{rem}

For the rest of the section let $0<a_1<\cdots<a_k$ and $g(x)=\sum_{m=1}^{k}\exp(a_mx)$. 
Let $\bar{g}(x)=g(-x)$. 
For any $\epsilon>0$ consider the function {$f_\epsilon \colon \R \to \R$ given by $f_\epsilon(x)=g(x)+\bar{g}(x)-\epsilon$.}

\begin{lem}\label{lem:exh1}
 For any $\epsilon>0$ the function $f_\epsilon$ is not $(k+1)$-positive semidefinite. 
\end{lem}

\begin{proof}
 {Let ${p}_m(x)=\exp(a_mx)$ for $m \in \{1,\ldots,k\}$. The Hankel matrix $H_{{p}_m}(x)$ of order $k+1$ equals the rank one matrix $$\exp(a_mx)\cdot v_m^{\top} \cdot v_m$$ where $v_m=(1,a_m,\ldots,a_m^k)$ for $m \in \{1,\ldots,k\}$.
 Let $w\in\R^{k+1}$ be a nonzero vector in the orthogonal complement of the subspace spanned by $v_1,\ldots,v_k$. We claim that $w^\top$ is in the kernel of $H_g(x)$ for all $x\in\R$. Indeed, we have for all $x\in\R$ that
 $$H_g(x)\cdot w^\top=\sum_{m=1}^kH_{{p}_m}(x)\cdot w^\top=\sum_{m=1}^k\exp(a_mx)\cdot v_m^{\top} \cdot \underbrace{v_m\cdot w^\top}_{=0}=0.$$
 Further, we claim that the first component of $w$ is nonzero. Indeed, assume otherwise and let $w'\in\R^k$ be such that $w=(0,w')$. Then $w'$ is in the orthogonal complement of the vectors $$(a_1,\ldots,a_1^k),\ldots,(a_k,\ldots,a_k^k).$$
 However, these vectors are linearly independent because $a_1,\dots,a_k$ are pairwise different and nonzero. This is a contradiction to our assumption that $w$ and thus $w'$ is nonzero. Therefore, we have shown that the first component of $w$ is nonzero. Without loss of generality, we may assume that it equals $1$. For $m \in \{1,\ldots,k\}$ let $b_m=-a_m$ and $u_m=(1,b_m,\ldots,b_m^k)$. Then, as above, we compute 
 $$wH_{\bar{g}}(x)w^\top=\sum_{m=1}^k\exp(-a_mx)\cdot\langle u_m,w\rangle^2$$
 which tends to zero for $x\to\infty$ because  $a_1,\dots,a_k$ are positive. Let $x_0\in\R$ be such that $wH_{\bar{g}}(x)w^\top<\epsilon$ for all $x>x_0$. Then we have for all $x>x_0$ that
 $$wH_{f_\epsilon}w^\top=wH_{\bar{g}}w^\top-\epsilon<0. $$ By \Cref{lem:hankelcrit} this proves the claim.}
\end{proof}

\begin{lem}\label{lem:exh2}
 For sufficiently small $\epsilon>0$ the function $f_\epsilon$ is $k$-positive definite. 
\end{lem}

\begin{proof}
Let $H_g(x)$ and $H_{\bar{g}}(x)$ be Hankel matrices of order $k$. They are 
  positive linear combinations of positive semidefinite rank one matrices so they are positive semidefinite for all $x\in\R$. {We can write $H_g(x)=A+B(x)$ where
    $A=\sum_{m=1}^k v_m^{\top}\cdot v_m$, $B(x)=\sum_{m=1}^k(\exp(a_mx)-1)\cdot v_m^{\top}\cdot v_m$ and $v_m=(1,a_m,\ldots
 ,a_m^{k-1})$. }
 {Similarly, we can write $H_{\bar g}(x)=A'+B'(-x)$ where
    $A'=\sum_{m=1}^k u_m^{\top}\cdot u_m$, $B'(x)=\sum_{m=1}^k(\exp(-a_mx)-1)\cdot u_m^{\top}\cdot u_m$ and $u_m=(1,-a_m,\ldots
 ,(-a_m)^{k-1})$.}
 We note that when $x\geq0$ we have that $B(x)$ is positive semidefinite because $\exp(a_mx)\geq 
 1$ for $x\geq0$ and similarly {$B'(-x)$} is positive semidefinite when $x\leq0$. In any case, $A$ {and $A'$ are} positive definite since the vectors $v_1,\ldots,v_k$ resp. $u_1,\ldots,u_k$ are linearly independent. Thus, it suffices to choose $\epsilon>0$ in such a way that both {$A-\epsilon \cdot \delta_1 \cdot \delta_1^{\top}$ and $A'-\epsilon \cdot \delta_1 \cdot \delta_1^{\top}$ are} positive definite, where $\delta_1$ is the first unit vector, since $\epsilon \cdot \delta_1 \cdot \delta_1^{\top}$ is the Hankel matrix of the constant function $\epsilon$. {Indeed, for $x\geq0$ we have that
 $$H_{f_\epsilon}(x)=(A-\epsilon \cdot \delta_1 \cdot \delta_1^{\top})+B(x)+H_{\bar{g}}(x)$$is positive semidefinite. Similarly, for $x\leq0$, we have that $$H_{f_\epsilon}(x)=(A'-\epsilon \cdot \delta_1 \cdot \delta_1^{\top})+B'(-x)+H_{{g}}(x)$$is positive semidefinite.
 }
\end{proof}

\begin{cor}\label{lem:exh3}
 For every $k\in\Z_{\geq0}$ there is a function $\R\to\R$ that is $k$-positive definite but not $k+1$-positive semidefinite.
\end{cor}

\begin{proof}
For sufficiently small $\epsilon>0$ the function $f_\epsilon$ has the desired properties.
\end{proof}

\begin{ex}
 Let $k=2$ and $0<a_1<a_2$. Then $$f_\epsilon(x)=\exp(a_1x)+\exp(a_2x)+\exp(-a_1x)+\exp(-a_2x)-\epsilon.$$For $x\geq0$ resp. $x\leq 0$ the Hankel matrix $H_f(x)$ is of the form $${\underbrace{\begin{pmatrix}2&\pm(a_1+a_2)\\\pm(a_1+a_2)&a_1^2+a_2^2
 \end{pmatrix}}_{=A\textnormal{ resp. } A'\textnormal{ in the proof of \ref{lem:exh2} }}+\begin{pmatrix}-\epsilon&0\\0&0
 \end{pmatrix}+C}(x)
 $$where ${C}(x)$ is positive semidefinite. For instance letting $a_1=\ln(2)$, $a_2=\ln(3)$,  
  and $\epsilon=\frac{1}{11}$ makes $H_{f_{\frac{1}{11}}}(x)$ positive definite for all $x\in\R$. {Indeed, in this case we have $$\det \begin{pmatrix}2-\epsilon&\pm(a_1+a_2)\\\pm(a_1+a_2)&a_1^2+a_2^2
 \end{pmatrix} =\frac{21}{11}\cdot(\ln(2)^2+\ln(3)^2)-(\ln(2)+\ln(3))^2\approx 0.011>0.$$
 Thus, $f_{\frac{1}{11}}$ is $2$-positive definite. However, the matrix $(f_{\frac{1}{11}}(x_i+x_j))_{1\leq i,j\leq 3}$ for $x_1=5, x_2=6, x_3=7$ equals
 $$ \begin{pmatrix}
 \frac{39956170693955}{665127936} & \frac{715125242123209}{3990767616} & \frac{12823220129727323}{23944605696} \\ &&&\\
 \frac{715125242123209}{3990767616} & \frac{12823220129727323}{23944605696} & \frac{230229525738486289}{143667634176} \\&&&\\
 \frac{12823220129727323}{23944605696} & \frac{230229525738486289}{143667634176} & \frac{4137070068201557555}{862005805056}
 \end{pmatrix}.$$Its determinant is $$-\frac{2277541160576348197}{107750725632}<0$$
 and thus the matrix is not positive semidefinite.} Therefore, $f_{\frac{1}{11}}$ is not 3-positive semidefinite.
\end{ex}

{In order to apply \Cref{thm:totimpcomp} consider the real closed field $H=\R[[\epsilon^\R]]$ of Hahn series over $\R$ in the ({infinitesimal}) indeterminate $\epsilon$ and with value group $\R$.}

\begin{cor}[Cor.~2.5 in \cite{Ave}]
 The cone {$K$} of $k\times k$ real symmetric positive semidefinite matrices has no semidefinite extended formulation with 
 LMIs of size less than $k$. 
\end{cor}

\begin{proof}
By \Cref{lem:exh3} there exists a 
function $f \colon \R\to\R$ that is $(k-1)$-positive definite  but not $k$-positive semidefinite. {By \Cref{rem:poskconvex} we can assume without loss of generality that $f(0)=1$.}
Then {$L_f \colon H\to H$} does not preserve $K$ by \Cref{thm:totimpcomp}(2).
But $L_f$ preserves all spectrahedra of size $k-1$ by~\Cref{lem:sdpdeg}, and hence it preserves all sets that have an semidefinite extended formulation with LMIs of size at most $k-1$. This shows that $K$ does not have such a formulation.
\end{proof}



In the same way as in the proof of \Cref{prop:pos-logic} we even obtain the following stronger statement.

\begin{cor}
The cone {K} of $k\times k$ real symmetric positive semidefinite matrices does not have a positive definition over the relational structure whose relations consist of all sets that can be described by an LMI of size {less than $k$} and the inequality relation $\neq$.
\end{cor}

\appendix

\section{Real closed separation}\label{sec:appendi}
We follow the line of arguments from \cite{rcsep}. We start with a lemma that is well-known which we include for the sake of completeness.

\begin{lem}\label{seplem}
Let $R$ be a real closed field and $I$ some finite index set. Let $V_1,V_2$ be finite-dimensional $R$-vector spaces. Let $\alpha_i \colon V_1 \to V_2$, $i\in I$, be some homogeneous polynomial maps  with the property that no non-trivial positive linear combination of images of the $\alpha_i$ is trivial, meaning that $$\sum_{j=1}^r \lambda_j\alpha_{i_j}(v_j)=0$$ for $r\in\N$, $i_j\in I$ and $\lambda_{j}>0$ implies that all $v_j \in V_{i_j}$ are zero. Then $$C := {\rm conv}({\cup_{i\in I}\alpha_i(V_1)})$$ is a closed convex cone. Moreover, if $b \not \in C$, then there exists an $R$-linear map $\phi \colon V_2 \to R$, such that $\phi(c)\geq 0$ for $c \in C$ and $\phi(b)<0$.
\end{lem}

\begin{proof}
 After choosing bases we can assume that $V_1=R^d$ and $V_2=R^n$.
 Let $$S=\{v\in R^d\mid ||v||_2=1\}$$ the unit sphere with respect to the euclidean norm. 
 Since each $\alpha_i$ is homogeneous, we have that $C$ is the conical hull of $\cup_{i\in I}\alpha_i(S)$. In particular, it is a convex cone. Thus it remains to prove that $C$ is closed. 
 Let $m=|I|$. We consider the map $$\Phi: S^{m\cdot n}\times\Delta_{m\cdot n}\to R^n,\, (v_{11},\ldots,v_{mn},\lambda_{11},\ldots,\lambda_{mn})\mapsto \sum_{j=1}^n\sum_{i=1}^m\lambda_{ij}\alpha_i(v_{ij})$$where $\Delta_{mn}\subset R^{mn}$ is the standard simplex. The image $B\subset R^n$ of $\Phi$ is semialgebraically compact as the image of a semialgebraically compact set under a continuous semialgebraic map. We claim that $$\Sigma=\{(v,b)\in R^n\times B\mid \exists \lambda\geq0: v=\lambda\cdot b\}$$is a closed subset of $R^n\times B$. Indeed, it can be written as $$\Sigma=\{(v,b)\in R^n\times B\mid \forall i,j: (v_ib_j=v_jb_i)\wedge (v_ib_i\geq0)  \}$$since $0\not\in B$ by assumption. Since $B$ is semialgebraically compact, the projection of $\Sigma$ onto the first coordinate, which equals to $C$ by Carath\'eodory's theorem for the conical hull, is closed as well. 
 Now the conical separation theorem implies the additional statement.
\end{proof}

Let $R$ be a real closed field and  $A$  a commutative $R$-algebra. For $g_1,\ldots,g_r\in A$ we denote  $${\rm U}_A(g_1,\ldots,g_r)=\{\alpha\in\Sper(A)\mid g_1(\alpha)>0,\ldots,g_r(\alpha)>0\}$$ where $\Sper(A)$ is the \emph{real spectrum} of $A$, see for example \cite[\S7]{RAG}. We further consider the \emph{quadratic module} ${\rm QM}_A(g_1,\ldots,g_r)$, i.e., the set of all $$s_0+s_1 g_1+\cdots+s_r g_r\in A$$where the $s_i\in A$ are sum of squares. We say that $A$ is \emph{real reduced} if $\sum_{i=1}^k a_i^2=0$ implies $a_1=\cdots=a_k=0$ for all $k \in \mathbb N$. 

\begin{thm}\label{thm:mainsep}
Let $R$ be a real closed field and let $A$ a real reduced commutative $R$-algebra. Let $b,g_1,\ldots,g_r\in A$ be such that the following holds:
\begin{enumerate}
    \item $b\not\in{\rm QM}_A(g_1,\ldots,g_r)$, 
    \item ${\rm U}_A(g_1,\ldots,g_r)$ is Zariski dense in $\Sper(A)$.
\end{enumerate}
Then there exists a real closed field $R'$ containing $R$ and an $R'$-linear map
$$\phi \colon A' = R' \otimes_R A \to R'$$
such that $\phi(1 \otimes b)<0$ and $\phi(c)\geq0$ for all $c\in{\rm QM}_{A'}(g_1,\ldots,g_r)$.
\end{thm}

\begin{proof}
{Since $A$ is real reduced, we have that $\Sper(A)\neq\emptyset$. If $g_i=0$ for some $i\in\{1,\dots,r\}$, then ${\rm U}_A(g_1,\ldots,g_r)=\emptyset$ contradicting our second assumption. Therefore, all the $g_i$ are nonzero.
In order to simplify notation we further let $g_0=1$.}
Let $H$ be a finite-dimensional $R$-subspace of $A$ containing $b,g_0,\ldots,g_r$. Consider the set
$$C(H) := \left\{\sum_{i=0}^rg_i\cdot\sum_{j=1}^{k_i} a_{ij}^2 \mid k_i \in \mathbb N, a_{ij} \in H \right\}.$$

Note that $C(H)$ is a convex cone inside the finite-dimensional subspace
$H^{(2)} := {\rm span}_R\{g_ia_1a_2 \in A \mid a_1,a_2 \in H,\, 0\leq i\leq r\}$ and $b \in H^{(2)}$ but $b \not \in C(H).$ Since $A$ is real-reduced and by condition $(2)$, we can apply \Cref{seplem} to the maps $$\alpha_i: H\to H^{(2)},\, a\mapsto a^2\cdot g_i.$$ Therefore, there exists an $R$-linear map $\phi_H \colon H^{(2)} \to R$ such that $\phi_H(c)\geq 0$ for all $c \in C(H)$ and $\phi_H(b)<0$. Let $\phi'_{H} \colon A \to R$ be an arbitrary linear extension of $\phi_H$ to $A$. Consider an ultrafilter $\mathcal U$ on the set of finite-dimensional subspaces of $A$, containing the filter generated by sets of the form $S_H:=\{H' \mid H \subset H' \}$ for $H$ a finite-dimensional subspace, and consider the ultra-product
$R':= \prod{}^{\mathcal U} R$ which is again a real closed field. There is a natural $R$-linear map $\phi' \colon A \to R'$ induced by all $\phi'_H$ for $H$ as above. 
Consider the $R'$-linear map 
$$\phi:=1_{R'} \otimes_R \phi' \colon R' \otimes_R A \to R'.$$ It is clear by construction that $\phi(1 \otimes b)<0$. 
We claim that $\phi$ is nonnegative on ${\rm QM}_{A'}(g_1,\ldots,g_r)$. It suffices to show that $\phi$ is nonnegative on elements of the form $a^2g_k$ for $a \in R' \otimes_R A$ and $0\leq k\leq r$. We write $a= \sum_{j=1}^l x_j \otimes y_j$ with $x_j \in R'$ and $y_j \in A.$ Let $H$ be a finite-dimensional subspace of $A$ containing $b,g_0,\ldots,g_r$ and $y_1,\dots,y_l$. Consider the matrix
$$(\phi(1 \otimes g_ky_iy_j))_{i,j=1}^{{l}} = (\phi'(g_ky_iy_j))_{i,j=1}^{{l}} \in M_{{l}}(R').$$
For $x=(x_1,\dots,x_l) \in R'^l$ and $y:=(g_ky_iy_j)_{i,j=1}^l$ we get
$\phi(a^2g_k)=x \phi'(y)x^{\top}$, so that it is sufficient to show that $\phi_H(y) \in M_{{l}}(R)$ is positive semi-definite for all such $H$. However, this is clear since for $x' \in R^l$, we have
$$x' \phi_H(y)x'^{\top} = \phi_H( (x'_1y_1+ \cdots +x'_ly_l)^2\cdot g_k) \geq 0$$ by construction. 
This finishes the proof. 
\end{proof}

\begin{cor}\label{cor:mainsep}
Let $R$ be a real closed field and let $A$ a real reduced commutative $R$-algebra. Let $b\in A$ not a sum of squares. 
Then there exists a real closed field $R'$ containing $R$ and an $R'$-linear map
$$\phi \colon R' \otimes_R A \to R'$$
which is nonnegative on sums of squares such that $\phi(1 \otimes b)<0$.
\end{cor}

\begin{rem}
 Let $X\subset R^n$ be some algebraic set and $A=R[X]$ its coordinate ring, i.e., the ring of all polynomial functions $X\to R$. Then $A$ is real reduced and the second condition of \Cref{thm:mainsep} is equivalent to the geometric condition that $$\{a\in X\mid g_1(a)>0,\ldots,g_r(a)>0\}$$is Zariski dense in $X$.
\end{rem}

\section{Preservation Theorems}\label{app:model}
The approach to study definable sets via structure-preserving maps can be used to obtain exact characterisations of existential positive, positive, and first-order definability. 
There are various formulations of such preservation theorems. Our formulation in Theorem~\ref{thm:mainpres} (in Section~\ref{sect:pres}) 
is closest to the setting of this article and we have been unable to find it in the literature, but it should be considered to be known. 

\subsection{Basic Terminology} 
\label{sect:basic} 
A \emph{signature} $\tau$ is a set of \emph{function symbols} and \emph{relation symbols}. Each symbol is equipped with an arity $k \in {\mathbb N}$. A \emph{$\tau$-structure} $\bA$ consists of a set $A$ (the \emph{domain}) and 
\begin{itemize}
    \item for each function symbol $f \in \tau$ of arity $k$ an operation $f^{\bA} \colon A^k \to A$, and
    \item for each relation symbol $R \in \tau$ of arity $k$ a  relation $R \subseteq A^k$. 
\end{itemize}
Note that the arity might in particular be $0$; function symbols of arity $0$ are also called \emph{constants}. 
We denote structures by $\bA$, $\bB$, $\bC$, and so on, and the corresponding underlying sets by the corresponding 
roman letters $A$, $B$, $C$, and so on. 
If $\bA$ and $\bB$ are structures with the same domain and $\bA$ is obtained from $\bB$ by dropping some of the relations and operations, then $\bA$ is called a \emph{reduct} of $\bB$, and $\bB$ is called an \emph{expansion} of $\bA$.

In the usual way we define \emph{$\tau$-terms} inductively; constant symbols from $\tau$ are $\tau$-terms,  variables are $\tau$-terms, and a function symbol of arity $k$ applied to a sequence of $\tau$-terms of length $k$ is a $\tau$-term. If $t$ is a $\tau$-term, we write $t(x_1,\dots,x_n)$ if all the variables that appear in $t$ are from $x_1,\dots,x_n$. 
An \emph{atomic $\tau$-formula} is 
an expression of the form
\begin{itemize}
\item $\bot$ (which stands for \emph{false}), or 
    \item $t_1 = t_2$ where $t_1$ and $t_2$ are $\tau$-terms, or 
\item $R(t_1,\dots,t_k)$ where $R \in \tau$ is a relation symbol of arity $k$ and $t_1,\dots,t_k$ are $\tau$-terms. 
\end{itemize}
A \emph{first-order $\tau$-formula} is built from atomic $\tau$-formulas in the usual inductive manner, using conjunction $\wedge$, disjunction $\vee$, negation {$\neg$}, universal quantification $\forall$, and existential quantification $\exists$. If $\phi$ is a first-order formula, we write $\phi(x_1,\dots,x_n)$ if all the variables that appear \emph{free} in $\phi$ (i.e., that are not bound by a universal or existential quantifier; for formal details, see~\cite{HodgesLong}) are from $x_1,\dots,x_n$. Formulas without free variables are called \emph{sentences}. 
If $\phi(x_1,\dots,x_k)$ is a $\tau$-formula and $a_1,\dots,a_k \in A$ we define 
$\bA \models \phi(a_1,\dots,a_n)$ if 
$(a_1,\dots,a_k)$ \emph{satisfies}  $\phi$ in $\bA$ (see~\cite{HodgesLong} for a precise definition). 
A $\tau$-formula $\phi(x_1,\dots,x_k)$ defines over $\bA$ a relation $R \subseteq A^k$, namely 
$$\big \{(a_1,\dots,a_k) \in A^k \mid \bA \models \phi(a_1,\dots,a_k) \big \}.$$ 
An $\tau$-formula $\phi(x_1,\dots,x_k)$ defines over $\bA$ an operation $f \colon A^{k-1} \to A$ if $\phi$ defines over $\bA$ the \emph{graph} of $f$, i.e., the relation $$
\{(a_1,\dots,a_k) \in A^k \mid f(a_1,\dots,a_{k-1}) = a_k \}.
$$
We say that $\bA$ is a \emph{first-order reduct} of $\bB$ if $\bA$ is a reduct of an expansion of $\bB$ by relations and operations that are first-order definable in $\bB$. 

\subsection{Preservation Theorems} 
\label{sect:pres} 
A first-order formula is called \emph{existential positive} if it is built without negation and universal quantification. Note that every existential positive formula is equivalent to a disjunction of \emph{primitive positive formulas}, i.e., formulas of the form
$$ \exists x_1,\dots,x_k (\psi_1 \wedge \dots \wedge \psi_m)$$
where $\psi_1,\dots,\psi_m$ are atomic formulas.

\begin{thm}\label{thm:mainpres}
Let $\bA$ be a structure. 
Then $\bA$ has an elementary extension $\bA^*$ such that the following holds. 
If $\bB$ is a first-order reduct of $\bA$ and $K \subseteq A^k$ is first-order definable over $\bA$, then
\begin{itemize}
    \item $K$ has an existential positive definition in $\bB$ if and only if 
    $K^*$ is preserved by all endomorphisms of $\bB^*$;
\item $K$ has a positive definition in $\bB$ if and only if 
$K^*$ is preserved by all surjective endomorphisms of $\bB^*$; 
    \item 
    $K$ has a first-order definition in $\bB$ if and only if $K^*$ is preserved by all automorphisms of $\bB^*$. 
\end{itemize}
Here, $\bB^*$ is the first-order reduct of $\bA^*$ defined by the same formulas as the first-order reduct $\bB$ of $\bA$, 
and $K^*$ is the relation defined over $\bA^*$ by the same first-order formula that defines $K$ over $\bA$. 
\end{thm}

We prove Theorem~\ref{thm:mainpres} in Section~\ref{sect:proof-pres}. 

\begin{rem}
In our setting, the real-closed field $\R$ plays the role of $\bA$.
For the results in Section~\ref{sect:groupring} we then 
work with the first-order reduct 
$\bB$ of $\bA$ that contains all spectrahedra as relations, to show that specific relations $K$ with a first-order definition in $\bA$ do not have a positive definition in $\bB$. 
In Section~\ref{sec:ave} we work with the first-order reduct $\bB$ of $\bA$ that contains all relations that are defined by LMIs of size $k$ to show that some specific relations $K$ with a first-order definition in $\bA$ do not have a positive definition in $\bB$. In both settings, Lemma~\ref{lem:sdpdeg} is a statement about the endomorphisms of $\bB^*$. 
\end{rem}

The elementary extension $\bA^*$ of $\bA$ in \Cref{thm:mainpres} will be a so-called \emph{special model} (see, e.g.,~\cite{HodgesLong}). In the context of preservation theorems often \emph{saturated structures} are used. 
Special structures are more technical to define than saturated structures. However, the existence of saturated structures requires the generalised continuum hypothesis which we would like to avoid. 

\subsection{Saturated and Special Structures} 
For $n \geq 0$, an
\emph{$n$-type}\index{$n$-type}\index{type} of a theory $T$ is a set $p$ 
of formulas with free variables $x_1,\dots,x_n$ such that $p \cup T$ 
is satisfiable.  
An \emph{$n$-type over a structure $\bA$} is an $n$-type of the first-order theory of $\fA$;
note that we do not admit parameters in these formulas (i.e., we allow constants to appear in the formulas of the type only if they belong to the signature).  
An $n$-type $p$ of $\bA$ is \emph{realised}\index{realised type} in $\bA$ if there exist $a_1,\ldots,a_n \in A$ such that $\bA \models \phi(a_1,\ldots,a_n)$ for each $\phi \in p$. 
For an infinite cardinal $\kappa$, a structure $\bA$ is \emph{$\kappa$-saturated}\index{$\kappa$-saturated} if, 
for all $\beta < \kappa$ and expansions $\bA'$ of $\bA$ by at most $\beta$ constants, every $1$-type of $\bA'$ is realised 
in $\bA'$. We say that an infinite structure $\bA$ is \emph{saturated} if it is $|A|$-saturated.
If $\lambda$ is a cardinal, then $\lambda^+$ denotes the successor cardinal of $\lambda$. 

\begin{thm}[Theorem 4.3.12 in~\cite{Marker}]
\label{thm:sat-ext}
Let $\tau$ be a signature and $\lambda \geq |\tau|$. 
Then every $\tau$-structure $\bB$ has a $\lambda^+$-saturated
elementary extension of cardinality $\leq |B|^\lambda$.
\end{thm}

So if we assume the Generalised
Continuum Hypothesis, then there
are saturated models of size $\kappa^+$ for all cardinals $\kappa$ (Corollary 10.2.4 in~\cite{HodgesLong}).

\begin{Def}
A structure $\bA$ of cardinality $\kappa \geq \omega$ is  
\emph{special} if $\bA$ is the union of an elementary chain $(\bB_{\lambda})_{\lambda < \kappa}$ and each $\bB_{\lambda}$ is $\lambda^+$-saturated. 
\end{Def}

Saturated structures are special, because we can write a saturated structure 
$\bA$ of cardinality $\kappa$ 
as $(\bB_{\lambda})_{\lambda < \kappa}$ where 
$\bB_{\lambda} = \bA$; since $\bA$ is $\lambda^+$-saturated for all $\lambda < \kappa$, this proves that $\bA$ is special. The converse is not true; see~\cite[Exercises 6 to 8 on page 514]{HodgesLong}. The advantage of special structures is that they always exist, independently from set-theoretic assumptions. 

\begin{thm}[Theorem 10.4.2.~in~\cite{HodgesLong}]
\label{thm:special-exists}
Every infinite structure has a special elementary extension. 
\end{thm}

\subsection{Proof of the Preservation Theorems}\label{sect:proof-pres}
The following is at the heart of the proof of \Cref{thm:mainpres}. 

\begin{thm}\label{thm:maps}
Let $\bA$ and $\bB$ be special structures of the same signature and cardinality. 
\begin{itemize}
\item If every first-order sentence that is true in $\bA$ is true in $\bB$, then there exists an isomorphism between $\bA$ and $\bB$. 
\item If every positive sentence that is true in $\bA$ is true in $\bB$, then there exists a surjective homomorphism from $\bA$ to $\bB$. 
\item If every existential positive sentence that is true in $\bA$ is true in $\bB$, then there exists a homomorphism from $\bA$ to $\bB$. 
\end{itemize}
\end{thm}
\begin{proof}
The first statement can be shown by a back and forth argument, and is Theorem 10.4.4.~in~\cite{HodgesLong}.
The other two statements can be shown analogously. 
\end{proof}

\begin{lem}[see Theorem 10.4.5.~in~\cite{HodgesLong}]
\label{lem:red-special}
First-order reducts of special structures are special. 
\end{lem}

\begin{proof}
[Proof of \Cref{thm:mainpres}]
In each of the three statements, the forward direction is immediate, since
homomorphisms preserve existential positive formulas, surjective homomorphisms preserve positive formulas, and automorphisms preserve first-order formulas. Moreover, if $K$ has a certain definition in $\bB$, then $K^*$ has the corresponding definition in $\bB^*$ since all the involved structures and sets are first-order definable in $\bA$ and $\bA^*$, respectively, and $\bA$ and $\bA^*$ have the same first-order theory. 

The three statements are well-known if $\bA$ is finite (see, e.g.,~\cite{KaluzninPoeschel}) so let us assume that $\bA$ is infinite. 
Let $\bA^*$ be a special elementary extension of $\bA$ which exists by \Cref{thm:special-exists}, 
and let $S$ be the first-order theory of $\bA$ (and $\bA^*$). 
Let $\bB$ be a first-order reduct of $\bA$ with signature $\tau$ and let $K \subseteq A^k$ be first-order definable over $\bA$. 
Let $\bB^*$ be the reduct of $\bA^*$ and $K^* \subseteq (A^*)^k$ be as in the statement of the theorem. By
\Cref{lem:red-special} we have that $\bB^*$ is special as well. 

To show the converse of the first statement, suppose that $K^*$ is preserved by all  endomorphisms of $\bB^*$. 
Let $\Psi$ be the set of all existential positive 
$\tau$-formulas $\psi(x_1,\dots,x_k)$ 
that hold on all tuples from $K^*$ in $\bB^*$. 
Let $\underline a$ be a tuple satisfying 
 $\Psi$ in $\bB^*$. Note that if there is no such tuple,
 then $\phi$ must be unsatisfiable in $\bB$. In this case, $\phi$ is equivalent to the existential positive formula $\bot$ and we are done. 

 Let $U$ be the set of all primitive positive $\tau$-formulas $\theta(\underline x)$ such that $\bB^* \models \neg \theta(\bar a)$. For any first-order $\tau$-formula $\eta$ 
 let $T(\eta)$ be the first-order formula 
 obtained from $\eta$ by replacing each atomic subformula of $\eta$ by its first-order definition over $\bA^*$.
 We claim that $$S \cup \{ T(\neg \theta) \mid \theta \in U \} \cup \{\phi \}$$ is satisfiable. Otherwise, by the compactness theorem of first-order logic there would be a finite subset $V$ of $U$ such that $S \cup \{ T(\neg \theta) \mid \theta \in V\} \cup \{ \phi \}$ is unsatisfiable.
Since $\neg T(\theta) = T(\neg \theta)$
we get that $S \cup \{\phi\}$ implies $\bigvee_{\theta \in V}  T(\theta)$. 
Thus, $\psi := \bigvee_{\theta \in V} \theta$ is an existential positive formula such that $S \cup \{\phi\} \models T(\psi)$.
But then $\psi \in \Psi$ and $\bB^* \models \psi(\underline a)$. 
This is in contradiction to the assumption that $\bB^* \models \neg \theta(\underline a)$ for all $\theta \in U$. 

Since $\bA^*$ is a special model of $S$, the $n$-type 
$\{ T(\neg \theta(\bar x)) \mid \theta \in U \} \cup \{\phi(\bar x) \}$ is realised by some $b \in (A^*)^k$ in $\bA^*$. 
Let $c_1,\dots,c_k$ be new constant symbols
and let $\bC$ be the $\tau \cup \{c_1,\dots,c_k\}$-expansion of $\bB^*$ where $\underline c^{\bC} = \underline b$.
Let $\bD$ be the $\tau \cup \{c_1,\dots,c_k\}$-expansion of $\bB^*$ where $\underline c^{\bC} = \underline a$. 
Clearly, $\bC$ and $\bD$ are special as well. 
Every primitive positive $(\tau \cup \{c_1,\dots,c_k\})$-sentence $\theta$
that is true in $\bD$ is also
true in $\bC$: otherwise, if $\theta$ were false in $\bC$, then write $\theta$ as $\theta(\underline c)$ for $\theta \in U$, 
in contradiction to the assumption that $\bD \models \{\neg \theta(\underline c) \; | \; \theta \in U\}$. 
Hence, since $\bC$ is special and by \Cref{thm:maps}, 
there exists a homomorphism from $\bD$ to
$\bC$. This homomorphism maps $\underline b$ to $\underline a$.
Since $\underline b$ satisfies $\phi$ in $\bA^*$ we have that
$\underline b \in K^*$, and since $K^*$ is preserved by endomorphisms of $\bB^*$, we have $\underline a \in K^*$. 
We conclude that $S \cup \{ T(\theta) \mid \theta \in \Psi\} \cup \{\neg \phi\}$ is unsatisfiable, and again by compactness there exists a finite subset $\Psi'$ of
$\Psi$ such that $S \cup \{T(\theta) \mid \theta \in \Psi'\} \cup \{\neg \phi\}$ is unsatisfiable. Note that $\psi := \bigwedge \Psi'$ is
an existential positive formula
and $S \models \forall \underline x (T(\psi) \Leftrightarrow \phi)$ 
which is what we wanted to show. 

The converse directions for statement two and three of the theorem can be shown analogously. 
\end{proof}

\bibliographystyle{plain}
\bibliography{references}

\end{document}